\documentclass[12pt]{amsart}
\usepackage{amsmath,amsfonts,amsthm,amsgen,amssymb,cite,mathrsfs}

\newcommand{\p}{\varphi}

\newcommand{\inv}{^{-1}}

\newcommand{\ov}{\overline}

\newcommand{\skel}[1]{^{(#1)}}

\newcommand{\dom}{\mathop{\boldsymbol d}}
\newcommand{\ran}{\mathop{\boldsymbol r}}

\newtheorem{Thm}{Theorem}
\newtheorem{Prop}[Thm]{Proposition}

{\theoremstyle{definition}
}
{\theoremstyle{remark}
}

{\theoremstyle{remark}
}
{\theoremstyle{remark}
}

{\theoremstyle{remark}
}
{\theoremstyle{remark}
}
{\theoremstyle{remark}
}
{\theoremstyle{remark}
\newtheorem*{Claim*}{Claim}}

\numberwithin{equation}{section}
\title[A note on projections in \'etale groupoid algebras]{A note on projections in \'etale groupoid algebras and diagonal preserving homomorphisms}
\author{Benjamin Steinberg}
\address[B.~Steinberg]{%
    Department of Mathematics\\
    City College of New York\\
    Convent Avenue at 138th Street\\
    New York, New York 10031\\
    USA}
\email{bsteinberg@ccny.cuny.edu}
\date{December 15, 2023}


\thanks{The author was supported by a Simons Foundation Collaboration Grant, award number 849561, and the Australian Research Council Grant DP230103184.}

\begin{document}

\begin{abstract}
Carlsen (Adv.~Math, 2018) showed that any $\ast$-homo\-mor\-phi\-sm between Leavitt path algebras over $\mathbb Z$ is automatically diagonal preserving and hence induces an isomorphism of boundary path groupoids.  His result works over conjugation-closed subrings of $\mathbb C$ enjoying certain properties.  In this paper, we characterize the rings considered by Carlsen as precisely those rings for which every $\ast$-homomorphism of algebras of Hausdorff ample groupoids is automatically diagonal preserving.  Moreover, the more general groupoid result has a simpler proof.
\end{abstract}

\maketitle

\section{Introduction}
The paper~\cite{cuntzsplice} caused quite a stir at the time because it showed that the Cuntz splice does not preserve $\ast$-isomorphism type of Leavitt path algebras over $\mathbb Z$ (or certain more general subrings of $\mathbb C$ closed under complex conjugation).  The Cuntz splice preserves isomorphism type of graph $C^*$-algebras and it is a major open question whether it preserves isomorphism type of complex Leavitt path algebras.  The result of~\cite{cuntzsplice} covered a fairly general class of graphs.    The key idea was to show that projections in Leavitt path algebras over such rings are quite restricted, and hence any $\ast$-isomorphism is forced to be diagonal-preserving.  This was extended to arbitrary Leavitt path algebras by Carlsen in~\cite{CarlsenLeavittZ}.  A series of results by various authors~\cite{reconstruct,CarlsenSteinberg,mydiagonal} shows that diagonal-preserving isomorphisms of algebras of Hausdorff ample groupoids forces, under mild hypotheses, an isomorphism of the corresponding groupoids.  Hence over rings in which $\ast$-isomorphisms are automatically diagonal-preserving, the groupoid is entirely encoded by its $\ast$-algebra.

The proofs in~\cite{cuntzsplice,CarlsenLeavittZ} unfortunately work with Leavitt path algebras as given by generators and relations, rather than as groupoid algebras, rendering them quite technical.  Here we give a very simple proof that for the kinds of rings considered in~\cite{cuntzsplice,CarlsenLeavittZ} (and only for those rings), every $\ast$-homomorphism of algebras of Hausdorff ample groupoids is automatically diagonal preserving.  We give several equivalent characterizations and prove some basic properties of such rings.

\section{The main result}
We follow the analytic conventions for groupoids, in particular, we identify objects and identity arrows.
An \emph{ample} groupoid $\mathscr G$ is a topological groupoid whose unit space $\mathscr G\skel 0$ is locally compact, Hausdorff and totally disconnected and whose range map $\ran\colon \mathscr G\to \mathscr G\skel 0$ is a local homeomorphism.  In this paper, all ample groupoids will be assumed Hausdorff.  In this case, the unit space $\mathscr G\skel 0$ is a clopen subspace of $\mathscr G$.  An open subset $U\subseteq \mathscr G$ is a (local) \emph{bisection} if $\dom|_U,\ran|_U$ are injective.  The compact bisections form a basis for the topology on $\mathscr G$.

If $R$ is a commutative ring with unit, the algebra $R\mathscr G$ of $\mathscr G$~\cite{mygroupoidalgebra} consists of the compactly  supported locally constant functions $f\colon \mathscr G\to R$ under convolution \[f\ast g(\gamma)=\sum_{\ran(\alpha)=\ran(\gamma)}f(\alpha)g(\alpha\inv \gamma).\]
The complex algebra $\mathbb C\mathscr G$ is a $\ast$-algebra with $f^\ast(\gamma) = \ov{f(\gamma\inv)}$.  From now on $R$ will always be a subring of $\mathbb C$ closed under complex conjugation.  The algebra $R\mathscr G$ is then a $\ast$-subalgebra of $\mathbb C\mathscr G$, and so we can talk about things like projections and unitaries in $R\mathscr G$.  For instance, if $U$ is a compact open subset of $\mathscr G\skel 0$, then the indicator function $1_U$ is a projection.  By a $\ast$-algebra homomorphism $\p\colon R\mathscr G\to R\mathscr H$, we mean an $R$-algebra homomorphism such that $\p(f^\ast)=\p(f)^\ast$ for all $f$.

Let us denote by $D(R\mathscr G)$ the subalgebra of $R\mathscr G$ consisting of functions supported on  $\mathscr G\skel 0$.  Then $D(R\mathscr G)$ is a commutative $\ast$-subalgebra of $R\mathscr G$, often referred to as the \emph{diagonal subalgebra}.  Note that $D(R\mathscr G)$ is spanned over $R$ by the projections $1_U$ with $U\subseteq \mathscr G\skel 0$ compact open.

A homomorphism $\p\colon R\mathscr G\to R\mathscr H$ of groupoid algebras is \emph{diagonal preserving} if $\p(D(R\mathscr G))\subseteq D(R\mathscr H)$.  We say that an isomorphism $\p$ is a diagonal preserving isomorphism if $\p$ and $\p\inv$ are diagonal preserving, or equivalently $\p$ is an isomorphism taking $D(R\mathscr G)$ onto $D(R\mathscr H)$.

An element $n\in R\mathscr G$ is a \emph{normalizer} of $D(R\mathscr G)$ if there is an element $n'$ with $nn'n=n$, $n'nn'=n'$ and $nD(R\mathscr G)n'\cup n'D(R\mathscr G)n\subseteq D(R\mathscr G)$.  It is easy to see that any $f\in R\mathscr G$ whose support is a bisection is a normalizer.  In~\cite{mydiagonal} $\mathscr G$ was defined to satisfy the local bisection hypothesis with respect to $R$ if the only normalizers are those with support a bisection.  For example, a group $G$ satisfies the local bisection hypothesis with respect to $R$ if and only if the group ring $RG$ has only trivial units. It was shown in~\cite{mydiagonal} that if there is a dense set of units $x\in \mathscr G\skel 0$ such that the group ring over $R$ of the isotropy group at $x$ of the interior of the isotropy bundle of $\mathscr G$ has only trivial units, then $\mathscr G$ satisfies the local bisection hypothesis.  This includes all boundary path groupoids of graphs and higher rank graphs.   The main theorem of~\cite{mydiagonal} admits the following as a special case.

\begin{Thm}\label{t:construct.from.bisection}
Let $R$ be an integral domain and let $\mathscr G,\mathscr G'$ be Hausdorff ample groupoids such that $\mathscr G$ satisfies the local bisection hypothesis. Then the following are equivalent.
\begin{enumerate}
\item There is an isomorphism $\p\colon \mathscr G\to \mathscr G'$.
\item There is a diagonal-preserving isomorphism $\Phi\colon R\mathscr G\to R\mathscr G'$ of $R$-algebras.
\end{enumerate}
\end{Thm}

The papers~\cite{cuntzsplice,CarlsenLeavittZ} investigated the case when every $\ast$-algebra homomorphism must automatically be diagonal preserving  in the setting of Leavitt path algebras.  We consider here the general case.

For $n\geq 1$, let $\mathcal R_n$ be the discrete groupoid associated to the universal equivalence relation on $\{1,\ldots, n\}$.   So $\mathcal R_n$ has these $n$ objects and a unique isomorphism $(i,j)$ from $j$ to $i$.  Multiplication follows the rule
\[(i,j)(k,\ell) = \begin{cases} (i,\ell), & \text{if}\ j=k\\ \text{undefined}, & \text{else}\end{cases}\]  and the inversion is given by $(i,j)\inv = (j,i)$.  It is not difficult to see that $R\mathcal R_n$ is $\ast$-isomorphic to $M_n(R)$ via $f\mapsto [f((i,j))]$ with inverse $A\mapsto f_A$ where $f_A((i,j))=A_{ij}$.  The diagonal subalgebra $D(R\mathcal R_n)$ consists of those functions supported on the diagonal and is sent via the above isomorphism onto the subalgebra $D_n(R)$ of diagonal matrices.  This explains the nomenclature.

The following theorem greatly generalizes and expands on~\cite{CarlsenLeavittZ,cuntzsplice}, where only Leavitt path algebras were considered.

\begin{Thm}\label{t:main}
Let $R$ be a subring of $\mathbb C$ closed under conjugation.  Then the following are equivalent.
\begin{enumerate}
\item For every $n\geq 1$, if $v\in R^n$ is a unit vector, then $v$ has exactly one nonzero entry, i.e., if $1=\sum_{i=1}^n|r_i|^2$, then only one $r_i\neq 0$.
\item If $r_1=\sum_{i=1}^n |r_i|^2$ with $r_1,\ldots, r_n\in R$, then $r_2=\cdots=r_n=0$.
\item If $\mathscr G$ is a Hausdorff ample groupoid, then each projection in $R\mathscr G$ belongs to the diagonal subalgebra $D(R\mathscr G)$.
\item Every $\ast$-homomorphism $R\mathscr G\to R\mathscr H$ of algebras of Hausdorff ample groupoids is diagonal preserving.
\item Every $\ast$-isomorphism $R\mathscr G\to R\mathscr H$ of algebras of Hausdorff ample groupoids is diagonal preserving.
\item For every $n\geq 1$, every unitary matrix in $GL_n(R)$ is monomial (i.e., has exactly one nonzero entry in every row and column).
\end{enumerate}
\end{Thm}
\begin{proof}
The first implication is in~\cite{CarlsenLeavittZ}, but we repeat the proof for the readers convenience.  If $r_1=\sum_{i=1}^n |r_i|^2$, then $r_1\geq 0$.  Therefore, \[|1-r_1|^2+\sum_{i=2}^n|r_i|^2+\sum_{i=1}^n |r_i|^2 = 1-2r_i+|r_1|^2+\sum_{i=2}^2|r_i|^2+\sum_{i=1}^n|r_i|^2 = 1.\]  If any $r_i\neq 0$, then $r_1>0$, and so we must have that $r_2=\cdots=r_n=0$ by (1).  Assume now (2) and let $f\in R\mathscr G$ be a projection.  Then $f=ff^*$.  Let $x\in \mathscr G\skel 0$.  Then \[f(x) = \sum_{\ran(\gamma)=x}f(\gamma)f^*(\gamma\inv)=\sum_{\ran(\gamma)=x}|f(\gamma)|^2.\]  Thus by (2), we must have that if $\ran(\gamma)=x$ and $\gamma\neq x$, then $f(\gamma)=0$.   Therefore, $f$ is supported on $\mathscr G\skel 0$ and hence $f\in D(R\mathscr G)$ as $\mathscr G$ is Hausdorff.  It is immediate that (3) implies (4) as $D(R\mathscr G)$ is spanned over $R$ by projections and each projection in $R\mathscr H$ is diagonal by (3).  Trivially (4) implies (5).  Recalling that $R\mathcal R_n\cong M_n(R)$ via a $\ast$-isomorphism taking $D(R\mathcal R_n)$ onto $D_n(R)$, for (5) implies (6) it suffices to observe that conjugation by a unitary matrix is a $\ast$-automorphism, and the normalizer in $GL_n(R)$ of $D_n(R)$ is the group of monomial matrices.   Finally, suppose that (6) holds and that $v\in R^n$ is a unit vector (which we view as a column vector).  Then $vv^*$ is a projection, and so $U=I-2vv^*$ is a self-adjoint unitary matrix.  Suppose that $v_i\neq 0\neq v_j$ with $i\neq j$.  Then $U_{ij} = -2v_i\ov{v_j}\neq 0$,  $U_{ii} = 1-2|v_i|^2$ and $U_{jj}= 1-2|v_j|^2$.  Thus if $U$ is monomial, then we must have $|v_i|^2=|v_j|^2=1/2$.  But then \[\begin{bmatrix}v_1 & -\ov v_2\\ v_2 & \ov v_1\end{bmatrix}\] is unitary and not monomial, contradicting (6).  Thus $v$ has a single nonzero entry.  This completes the proof.
\end{proof}

We remark that it was claimed without proof in~\cite{CarlsenLeavittZ} that there are rings satisfying (2) but not (1).  In fact, it is easy to see directly that (2) implies (1) since if $1=\sum_{i=1}^n|r_i|^2$, where without loss of generality $r_1\neq 0$, then $|r_1|^2 = \sum_{i=1}^n|r_1r_i|^2$, and so $r_i=0$ for $i\geq 2$ by (2).

Those rings satisfying the equivalent conditions of Theorem~\ref{t:main} were called \emph{kind} in~\cite{CarlsenLeavittZ} and were said to have a unique partition of the unit in~\cite{cuntzsplice}.  We prefer `kind'.
Rings with the property that $c_1^2+\cdots+c_n^2=1$ implies $c_i=\pm 1$ for a unique $i$, and $c_j=0$ otherwise, were studied in~\cite{L-ring} under the name $L$-rings.  Many of the observations in that paper about $L$-rings apply to kind rings.  Note that a subring of $\mathbb R$ is kind if and only if it is an $L$-ring.  Any $L$-ring which is a subring of the complex numbers closed under complex conjugation must be kind.  There are however, kind rings that are not $L$-rings.  For instance $\mathbb Z[i]$ is kind (see Proposition~\ref{p:propsofkind} below), but $2^2+i^2+i^2+i^2=1$, so it is not an $L$-ring.

\begin{Prop}\label{p:propsofkind}
Let $R$ be a subring of $\mathbb C$ closed under complex conjugation and let $F$ be the field of fractions of $R$.
\begin{enumerate}
\item If $R$ is kind, then $1/n\notin R$ for all $n\geq 2$.
\item If $|r|<1$ implies $r=0$ for all $r\in R$, then $R$ is kind.  In particular, $\mathbb Z$ is kind, and if $a\in [1,\infty)$, then $\mathbb Z[ia]$ is kind.
\item A directed union of kind rings is kind.
\item If $R$ is kind and $B\subseteq \mathbb R$ is algebraically independent over $F$, then $R[B]$ is kind.  In particular, $\mathbb Z[\pi]$ and $\mathbb Z[e]$ are kind.
\item If $R$ is kind and $n\in \mathbb Z$ with $\sqrt{n}\notin F$, then $R[\sqrt n]$ is kind.  In particular, if $R$ is integrally closed in $F$ and $\sqrt{n}\notin R$, then $R[\sqrt{n}]$ is kind.
\end{enumerate}
\end{Prop}
\begin{proof}
If $1/n\in R$, then $(1/n,1/n,\ldots,1/n)\in R^{n^2}$ is a unit vector, and so $R$ is not kind.  This proves (1).  The second item is clear since if $1=|r_1|^2+\cdots+|r_n|^2$, then $|r_i|^2\geq 1$ for at most one value of $i$.  Item (3) is clear since a unit vector has finitely many entries.   For (4), first note that $R[B]$ is closed under complex conjugation since $B\subseteq \mathbb R$. We may assume by (3) that $B$ is finite and then by induction it suffices to handle the case $R[a]$ where $a\in \mathbb R$ is transcendental over $F$.  Suppose that $(f_1(a),\ldots, f_n(a))\in R[a]^n$ is a unit vector with the $f_i\in R[x]$.  Then since $a\in \mathbb R$, we have that $|f_i(a)|^2 = (f_i\ov f_i)(a)$ and $f_i\ov f_i$ is a polynomial over $R$ with real coefficients and with leading coefficient the square of the absolute value of the leading coefficient of $f_i$.  Therefore, if  some $f_i$ is a nonconstant polynomial, then $g(x)=f_1\ov f_1+\cdots +f_n\ov f_n-1$ is a nonzero polynomial of degree twice the maximum degree of the $f_i$ with $g(a)=0$.  This contradicts that $a$ is transcendental over $F$.  Thus $f_1,\ldots, f_n$ are constant polynomials, and so $(f_1(a),\ldots, f_n(a))\in R^n$ and hence has exactly one nonzero entry since $K$ is kind.  Finally, if $\sqrt{n}\notin F$ and $(a_1+b_1\sqrt{n},\ldots, a_m+b_m\sqrt{n})\in R[\sqrt{n}]^m$ is a unit vector, then \[1=\sum_{i=1}^m (|a_i|^2+|b_i|^2|n|)+\sqrt{n}\sum_{i=1}^m (\pm a_i\ov b_i+\ov a_ib_i).\]  Since $\sqrt{n}\notin F$, we must have $\sum_{i=1}^m (\pm a_i\ov b_i+\ov a_ib_i)=0$.  Since $|n|$ is a positive integer, we deduce using that $R$ is kind that there is a unique $i$ with either $a_i\neq 0$ or $b_i\neq 0$.  We conclude that $R[\sqrt{n}]$ is kind.
\end{proof}


\begin{thebibliography}{1}

\bibitem{reconstruct}
P.~Ara, J.~Bosa, R.~Hazrat, and A.~Sims.
\newblock Reconstruction of graded groupoids from graded {S}teinberg algebras.
\newblock {\em Forum Math.}, 29(5):1023--1037, 2017.

\bibitem{CarlsenLeavittZ}
T.~M. Carlsen.
\newblock {$\ast$}-isomorphism of {L}eavitt path algebras over {$\Bbb Z$}.
\newblock {\em Adv. Math.}, 324:326--335, 2018.

\bibitem{CarlsenSteinberg}
T.~M. Carlsen and J.~Rout.
\newblock Diagonal-preserving graded isomorphisms of {S}teinberg algebras.
\newblock {\em Commun. Contemp. Math.}, 20(6):1750064, 25, 2018.

\bibitem{L-ring}
H.~Ishibashi.
\newblock Structure of the orthogonal group {$O_n(V)$} over {$L$}-rings.
\newblock {\em Linear Algebra Appl.}, 390:357--368, 2004.

\bibitem{cuntzsplice}
R.~Johansen and A.~P.~W. S{\o}rensen.
\newblock The {C}untz splice does not preserve {$\ast$}-isomorphism of
  {L}eavitt path algebras over {$\Bbb{Z}$}.
\newblock {\em J. Pure Appl. Algebra}, 220(12):3966--3983, 2016.

\bibitem{mygroupoidalgebra}
B.~Steinberg.
\newblock A groupoid approach to discrete inverse semigroup algebras.
\newblock {\em Adv. Math.}, 223(2):689--727, 2010.

\bibitem{mydiagonal}
B.~Steinberg.
\newblock Diagonal-preserving isomorphisms of \'{e}tale groupoid algebras.
\newblock {\em J. Algebra}, 518:412--439, 2019.

\end{thebibliography}

\end{document}